\documentclass[11pt,a4paper]{amsart}
\usepackage{amssymb,amsmath,layout}
\usepackage{amssymb,amsmath,epsfig,graphics}

\theoremstyle{plain}
\newtheorem{thrm}{Theorem}[section]
\newtheorem{lemma}[thrm]{Lemma}

\newtheorem{cor}[thrm]{Corollary}
\newtheorem{rmrk}[thrm]{Remark}

\begin{document}

\newcommand{\SL}{\mathcal L^{1,p}( D)}
\newcommand{\Lp}{L^p( Dega)}
\newcommand{\CO}{C^\infty_0( \Omega)}
\newcommand{\Rn}{\mathbb R^n}
\newcommand{\Rm}{\mathbb R^m}
\newcommand{\R}{\mathbb R}
\newcommand{\Om}{\Omega}
\newcommand{\Hn}{\mathbb H^n}
\newcommand{\aB}{\alpha B}
\newcommand{\eps}{\ve}
\newcommand{\BVX}{BV_X(\Omega)}
\newcommand{\p}{\partial}
\newcommand{\IO}{\int_\Omega}
\newcommand{\bG}{\boldsymbol{G}}
\newcommand{\bg}{\mathfrak g}
\newcommand{\bz}{\mathfrak z}
\newcommand{\bv}{\mathfrak v}
\newcommand{\Bux}{\mbox{Box}}
\newcommand{\e}{\ve}
\newcommand{\X}{\mathcal X}
\newcommand{\Y}{\mathcal Y}
\newcommand{\W}{\mathcal W}

\numberwithin{equation}{section}

\newcommand{\RN} {\mathbb{R}^N}
\newcommand{\Sob}{S^{1,p}(\Omega)}
\newcommand{\Dxk}{\frac{\partial}{\partial x_k}}
\newcommand{\Co}{C^\infty_0(\Omega)}
\newcommand{\Je}{J_\ve}
\newcommand{\beq}{\begin{equation}}
\newcommand{\bea}[1]{\begin{array}{#1} }
\newcommand{\eeq}{ \end{equation}}
\newcommand{\ea}{ \end{array}}
\newcommand{\eh}{\ve h}
\newcommand{\Dxi}{\frac{\partial}{\partial x_{i}}}
\newcommand{\Dyi}{\frac{\partial}{\partial y_{i}}}
\newcommand{\Dt}{\frac{\partial}{\partial t}}
\newcommand{\aBa}{(\alpha+1)B}
\newcommand{\GF}{\psi^{1+\frac{1}{2\alpha}}}
\newcommand{\GS}{\psi^{\frac12}}
\newcommand{\HFF}{\frac{\psi}{\rho}}
\newcommand{\HSS}{\frac{\psi}{\rho}}
\newcommand{\HFS}{\rho\psi^{\frac12-\frac{1}{2\alpha}}}
\newcommand{\HSF}{\frac{\psi^{\frac32+\frac{1}{2\alpha}}}{\rho}}
\newcommand{\AF}{\rho}
\newcommand{\AR}{\rho{\psi}^{\frac{1}{2}+\frac{1}{2\alpha}}}
\newcommand{\PF}{\alpha\frac{\psi}{|x|}}
\newcommand{\PS}{\alpha\frac{\psi}{\rho}}
\newcommand{\ds}{\displaystyle}
\newcommand{\Zt}{{\mathcal Z}^{t}}
\newcommand{\XPSI}{2\alpha\psi \begin{pmatrix} \frac{x}{|x|^2}\\ 0 \end{pmatrix} - 2\alpha\frac{{\psi}^2}{\rho^2}\begin{pmatrix} x \\ (\alpha +1)|x|^{-\alpha}y \end{pmatrix}}
\newcommand{\Z}{ \begin{pmatrix} x \\ (\alpha + 1)|x|^{-\alpha}y \end{pmatrix} }
\newcommand{\ZZ}{ \begin{pmatrix} xx^{t} & (\alpha + 1)|x|^{-\alpha}x y^{t}\\
     (\alpha + 1)|x|^{-\alpha}x^{t} y &   (\alpha + 1)^2  |x|^{-2\alpha}yy^{t}\end{pmatrix}}
\newcommand{\norm}[1]{\lVert#1 \rVert}
\newcommand{\ve}{\varepsilon}

\title[Modica type gradient estimates, etc.]{Modica type gradient estimates for reaction-diffusion equations and a parabolic counterpart of a conjecture of De Giorgi}

\dedicatory{Dedicated to Ermanno Lanconelli, on the occasion of his birthday, with deep affection and admiration}

\author{Agnid Banerjee}
\address{Department of Mathematics\\University of California, Irvine \\
CA- 92697} \email[Agnid Banerjee]{agnidban@gmail.com}
\thanks{First author was supported in part by the second author's
NSF Grant DMS-1001317 and by a postdoctoral grant of the Institute Mittag-Leffler}

\author{Nicola Garofalo}
\address{Dipartimento di Ingegneria Civile, Edile e Ambientale (DICEA) \\ Universit\`a di Padova\\ 35131 Padova, ITALY}
\email[Nicola Garofalo]{rembdrandt54@gmail.com}
\thanks{Second author was supported in part by NSF Grant DMS-1001317 and by a grant of the University of Padova, ``Progetti d'Ateneo 2013"}

%
%
%
\keywords{}
\subjclass{}

\maketitle

\begin{abstract}
We  continue the study of Modica type gradient estimates for non-homogeneous parabolic equations  initiated in \cite{BG}. First, we  show that for the parabolic minimal surface equation with a  semilinear  force term if a certain gradient estimate is satisfied at  $t=0$, then it holds for all later times $t>0$. We  then establish analogous  results for  reaction-diffusion equations such as \eqref{e0} below in $\Om \times [0, T]$, where  $\Om$ is an epigraph such that the mean curvature  of   $\partial \Om$ is nonnegative. 

We then turn our attention to settings where such gradient estimates are valid without any a priori information on whether the estimate holds at some earlier time. Quite remarkably (see Theorem \ref{main3}, Theorem \ref{main5} and Theorem \ref{T:ricci}), this is is true for $\Rn \times (-\infty, 0]$ and  $\Om \times  (-\infty, 0]$, where $\Om$ is  an epigraph satisfying  the geometric  assumption  mentioned above, and for $M \times (-\infty,0]$, where $M$ is a connected, compact Riemannian  manifold with nonnegative Ricci tensor. As a  consequence  of the gradient estimate \eqref{mo2}, we  establish a rigidity  result (see Theorem \ref{main6} below) for solutions to \eqref{e0} which is the   analogue  of  Theorem 5.1   in \cite{CGS}. Finally, motivated  by Theorem \ref{main6}, we  close the paper  by proposing a parabolic  version of  the famous conjecture of De Giorgi also known as the $\ve$-version of the Bernstein theorem. 
\end{abstract}

\section{introduction}\label{S:intro}

In his pioneering paper \cite{Mo} L. Modica proved that if $u$ is a (smooth) bounded entire solution of the semilinear Poisson equation $\Delta u = F'(u)$ in $\Rn$, with nonlinearity $F\ge 0$, then $u$ satisfies the a priori gradient bound 
\begin{equation}\label{mo}
|Du|^2 \le 2 F(u).
\end{equation}
With a completely different approach from Modica's original one, this estimate was subsequently extended in \cite{CGS} to nonlinear equations in which the leading operator is modeled either on the $p$-Laplacian $ \operatorname{div}(|Du|^{p-2} Du)$, or on the minimal surface operator $\operatorname{div}((1+|Du|^2)^{-1/2}Du)$, and later to more general integrands of the calculus of variations in \cite{DG}. More recently, in their very interesting paper \cite{FV} Farina and Valdinoci have extended the Modica estimate \eqref{mo} to domains in $\Rn$ which are epigraphs whose boundary has nonnegative mean curvature, and to compact manifolds having nonnegative Ricci tensor, see \cite{FV4}, and also the sequel paper with Sire \cite{FSV}.


It is by now well-known, see \cite{Mo}, \cite{CGS}, \cite{AC}, \cite{DG}, that, besides its independent interest, an estimate such as \eqref{mo} implies Liouville type results, monotonicity properties of the relevant energy and it is also connected to a famous conjecture of De Giorgi (known as the $\ve$-version of the Bernstein theorem) which we discuss at the end of this introduction and in Section \ref{S:dg} below, and which nowadays still constitutes a largely unsolved problem.  

In the present paper we study Modica type gradient estimates for solutions of some nonlinear parabolic equations in $\Rn$ and, more in general, in complete Riemannian manifolds with nonnegative Ricci tensor, and in unbounded domains satisfying the above mentioned geometric assumptions in  \cite{FV}.   
In the first part of the paper we  continue the  study  initiated in the recent work \cite{BG}, where we considered the following  inhomogeneous  variant of the normalized $p$-Laplacian evolution  in $\Rn  \times [0,T]$,
\begin{equation}\label{1}
 |Du|^{2-p}\left\{\operatorname{div}(|Du|^{p-2}Du) - F'(u)\right\} = u_t,\ \ \ \ \  1 < p \leq 2.
\end{equation}
In \cite{BG} we proved that if  a bounded  solution $u$ of \eqref{1} belonging to a certain class $H$ (see \cite{BG} for the relevant definition) satisfies the following gradient estimate at  $t=0$ for a.e. $x \in \Rn$,
\begin{equation}\label{modp}
|Du(x, t)|^{p} \leq \frac{p}{p-1} F(u(x, t)),
\end{equation}
then such estimate continues to hold at any given time  $t > 0$. On the function $F$ we assumed that $F \in C^{2, \beta}_{loc}(\R)$ and $F\ge 0$. These same assumptions will be assumed throughout this whole paper.

In Section \ref{S:modica} we show that a similar result  is true for  the following  inhomogeneous  variant of the minimal surface parabolic equation
\begin{equation}\label{e:p}
(1+ |Du|^{2})^{1/2} \left\{\operatorname{div}\left(\frac{Du}{(1+ |Du|^2)^{1/2}}\right) - F'(u)\right\} = u_t, 
\end{equation}
see Theorem \ref{main1} below. The equation \eqref{e:p} encompasses two  types of equations: when $F(u) = 0$ it represents the equation of motion by mean curvature studied in \cite{EH}, whereas when $u(x,t) = v(x)$, then \eqref{e:p} corresponds to the steady state which is prescribed mean curvature equation.

In Section \ref{S:fv} we establish  similar results for the reaction diffusion equation in $\Om \times [0, T]$ 
\begin{equation}\label{e0}
\Delta u = u_t  + F'(u),
\end{equation}
where now $\Om$ is an epigraph and the mean  curvature of $\partial \Om$ is nonnegative. Theorem \ref{main2} below constitutes the parabolic  counterpart  of  the cited result  in \cite{FV} for the following  problem 
\begin{equation}\label{e1}
\begin{cases}
\Delta u=  F'(u),\  \ \ \text{in}\  \Om,
\\
u=0\ \text{on}\ \ \partial \Om,\  \ \  u\geq 0\ \text{on}\ \Om.
\end{cases}
\end{equation}
In that paper the authors proved that a  bounded  solution $u$  to \eqref{e1}  satisfies the Modica estimate \eqref{mo},
provided that the mean curvature of $\p \Om$ be nonnegative.

In Section  \ref{S:moremods} we  turn our  attention to settings  where  global versions of such estimates  for  solutions to \eqref{e0} can be established, i.e., when there is no a priori information on  whether   such an  estimate  hold at some earlier time $t_0$.  In Theorems \ref{main3} and  \ref{main5} we show that, quite  remarkably, respectively in the case  $\Rn \times (-\infty, 0]$ and $\Om \times (-\infty, 0]$, where  $\Om$ is an epigraph that satisfies   the geometric  assumption mentioned above, the a priori gradient estimate 
\begin{equation}\label{mo2}
|Du(x,t)|^2 \le 2 F(u(x,t))
\end{equation}
holds globally on a bounded solution $u$ of \eqref{e0}.  

In Section \ref{S:ricci} we establish a parabolic generalization of the result in \cite{FV4}, but in the vein of our global results in Section \ref{S:moremods}. In Theorem \ref{T:ricci} we prove that if
$M$ is a  compact  Riemannian manifold  with Ric $\ge 0$, with Laplace-Beltrami $\Delta$, then any bounded entire solution $u$ to \eqref{e0} in $M\times (-\infty,0]$  satisfies \eqref{mo2}. It remains to be seen whether our result, or for that matter the elliptic result in \cite{FV4}, remain valid  when $M$ is only  assumed  to be complete, but not compact.

Finally in Section \ref{S:dg}, as a  consequence of the a priori estimate \eqref{mo2} in Section \ref{S:moremods}, we  establish  an analogue of Theorem 5.1  in \cite{CGS}  for  solutions to \eqref{e0} in $\Rn \times (-\infty, 0]$. More precisely,  in Theorem \ref{main6} below we show that if  the equality in \eqref{mo2} holds at some $(x_0, t_0)$, then there exists a function $g\in C^2(\R)$, $a \in \Rn$, and $\alpha\in \R$, such that
\begin{equation}
u(x, t)= g(<a, x> + \alpha).
\end{equation}
In particular, $u$ is independent of time and the level sets of $u$  are vertical hyperplanes in $\Rn\times (-\infty,0]$. This result suggests a parabolic version of the famous conjecture of De Giorgi  (also known as the $\ve$-version of the Bernstein theorem for minimal graphs) which asserts that entire solutions to
\begin{equation}
\Delta u = u^3 - u,
\end{equation}
such that $|u|\le 1$ and $\frac{\p u}{\p x_n} >0$, must be one-dimensional, i.e., must 
have level sets which are hyperplanes, at least in dimension $n\le 8$, see \cite{dG}. We recall that the conjecture of De Giorgi has been fully solved for $n=2$ in \cite{GG} and $n=3$ in \cite{AC}, and it is known to fail for $n\ge 9$, see \cite{dPKW}. Remarkably, it is still an open question for $4 \le n \le 8$. Additional fundamental progress on De Giorgi's conjecture is contained in the papers \cite{GG2}, \cite{Sa}.  For results  concerning the $p$-Laplacian version of De Giorgi's conjecture, we refer the reader to the interesting paper \cite{SSV}.  For further  results, the  state of art  and  recent  progress on De Giorgi's conjecture, we refer to \cite{CNV}, \cite{FV2}, \cite{FSV1} and the references therein.

In Section \ref{S:dgc} motivated by our Theorem \ref{main6} below, we close the paper by proposing a parabolic version of  De Giorgi's conjecture. It is our hope that it will stimulate interesting further research.

\medskip

\noindent \textbf{Acknowledgment:} The paper was finalized  during the first author's stay at the Institut Mittag-Leffler during the semester-long program \emph{Homogenization and Random Phenomenon}. The first  author would like to thank the Institute and the organizers of the program for the kind hospitality and the  excellent working conditions. We would like to thank Matteo Novaga for kindly bringing to our attention that Conjecture 1 at the end of this paper is violated by the traveling wave solutions in \cite{CGHNR} and \cite{G} and for suggesting the amended Conjecture 2.

\section{Forward Modica  type estimates in $\Rn \times [0, T]$ for the generalized motion by mean curvature equation}\label{S:modica}

In \cite{CGS} it was proved that if $u \in C^{2}(\Rn) \cap L^{\infty}(\Rn)$ is a solution to
\begin{equation}
 \operatorname{div}\left(\frac{Du}{(1+ |Du|^2)^{1/2}}\right)= F'(u),
\end{equation}
such that  $|Du| \leq C$, then the following Modica type gradient estimate holds
\begin{equation}
\frac{(1+ |Du|^2)^{1/2} -1}{(1+ |Du|^2)^{1/2}} \leq F(u). 
\end{equation}
In Theorem \ref{main1} below we generalize this result to the parabolic minimal surface equation \eqref{e:p}.  Such result also provides the counterpart of the above cited main result \eqref{1} in \cite{BG} for the normalized parabolic $p$-Laplacian \eqref{modp}. Henceforth, by $v \in C^{2,1}_{loc}$, we mean that $v$ has continuous derivatives of up to  order two in  the $x$ variable and up to order one in the $t$  variable. We would also  like to mention that  unlike the case when $F=0$,  further requirements on $F$ need to be imposed to  ensure that a  bounded  solution to \eqref{e:p} has bounded gradient, see e.g. Theorem 4 in \cite{LU1}. This is why an $L^{\infty}$  gradient bound  is assumed in the hypothesis of the next theorem.
 
We recall that throughout the whole paper we assume that $F \in C^{2, \beta}_{loc}(\R)$ for some $\beta > 0$,  and that $F\ge 0$. 

\begin{thrm}\label{main1}
For a given $\ve> 0$, let $u \in C^{2,1}_{loc}(\Rn \times [0, T]) \cap L^{\infty}(\Rn \times (-\ve, T])$  be a classical solution to \eqref{e:p} in $\Rn \times [0,T]$ such that $|Du| \leq C$. If $u$ satisfies the following  gradient estimate
\begin{equation}\label{est1}
\frac{(1+ |Du|^2)^{1/2} -1}{(1+ |Du|^2)^{1/2}} \leq F(u) 
\end{equation}
at $t=0$, then $u$ satisfies  \eqref{est1} for all  $t>0$.
\end{thrm}

\begin{proof}
Since $|Du| \leq C$ and $F \in C^{2,\beta}_{loc}$, it follows from the Schauder regularity theory of  uniformly parabolic  non-divergence  equations (see Chapters 4 and 5 in \cite{Li}), that $u \in H_{3+\alpha}(\Rn \times [0, T])$  for some  $\alpha>0$ which depends on $\beta$ and the bounds on $u$ and $Du$ (see Chapter 4 in \cite{Li} for the relevant notion). Now  we let 
\begin{equation}\label{def}
\phi (s)= (s^2 + 1)^{1/2},\ \ \ \ \ \ s\in \R.
\end{equation}
With this notation we have that  $u$ is a classical solution to 
\begin{equation}\label{approx}
\operatorname{div} (\phi' (|Du|^{2}) Du)= \phi' (|Du|^{2}) u_t + F'(u)
\end{equation}
Now  given that $u \in H_{3+\alpha}(\Rn \times  [0, T])$,   one can repeat the arguments as in the proof of Theorem 5.1 in \cite{BG}  with $\phi$ as in \eqref{def}. We nevertheless provide the details for the sake of completeness and also because the corresponding  growth of $\phi$ in $s$ is quite different from  the one in Theorem 5.1 in \cite{BG}. Let
\begin{equation}
\xi(s)= 2s \phi'(s) - \phi(s),
\end{equation}
and define $\Lambda = \xi'$. We  also define $P$ as follows
\begin{equation}\label{P:1}
P(u, x,t) = \xi(|Du(x,t)|^{2}) - 2 F(u(x,t)).
\end{equation}
With $\phi$ as in \eqref{def} above, we have that 
\begin{equation}
P= 2\frac{(1+ |Du|^2)^{1/2} -1}{(1+ |Du|^2)^{1/2}} - 2F(u).
\end{equation}
We note that the hypothesis that \eqref{est1} be valid at $t=0$ can be reformulated by saying that $P(\cdot,0) \leq 0$. We next write  \eqref{approx} in the following manner
\begin{equation*}
a_{ij} (Du)\ u_{ij}= f(u) + \phi'\  u_{t},
\end{equation*}
where  for $\sigma \in \Rn$ we have let
\begin{equation}\label{nondiv}
a_{ij}(\sigma)= 2\phi'' \sigma_{i}\ \sigma_{j} + \phi' \delta_{ij}.
\end{equation}
Therefore, $u$ satisfies
\begin{equation}\label{e:650}
d_{ij}\  u_{ij} = \frac{f}{\Lambda} + \frac{\phi'}{\Lambda} u_{t},
\end{equation}
where $d_{ij}= \frac{a_{ij}}{\Lambda}$.
By differentiating \eqref{nondiv} with respect to $x_{k}$, we obtain
\begin{equation}\label{f:135}
(a_{ij}\ (u_k)_{i})_{j}= f'\ u_k  + \phi' \ u_{tk} + 2 \phi''\ u_{hk}\ u_{h}\ u_{t}.
\end{equation}
From the definition of $P$  in \eqref{P:1} we have,
\begin{equation}\label{P1}
P_{i} = 2 \Lambda u_{ki}\  u_{k} - 2 f\ u_{i},\ \ \ P_{t} = 2 \Lambda u_{kt}\ u_{k} - 2f\  u_{t}.
\end{equation}

We  now consider the following auxiliary function
\begin{equation*}
w = w_R = P -  \frac{M}{R} \sqrt{|x|^2 + 1} -   \frac{ct}{R^{1/2}},
\end{equation*}
where $R > 1$ and  $M$, $c$ are to be  determined subsequently. 
Note that $P \ge w$ for $t\ge 0$.
Consider the cylinder  $ Q_{R}=  B( 0 , R) \times  [0, T]$. One can see that  if $M$  is   chosen large enough,  depending on the $L^{\infty}$ norm of $u$ and its first derivatives, then $ w < 0 $ on the lateral boundary of $Q_{R}$. In this situation we see that if $w$ has a strictly positive maximum at a point $(x_0, t_0)$,
then such point cannot be on the parabolic boundary of $Q_R$. In fact, since $w<0$ on the lateral boundary, the point cannot be on such set. But it cannot be on the bottom of the cylinder either since, in view of \eqref{est1}, at $t=0$ we have  $ w(\cdot,0) \leq P(u(\cdot,0)) \leq 0$.

Our objective is to prove the following claim: 
\begin{equation}\label{claim}
w \le K \overset{def}{=} R^{- \frac 12},\ \ \ \ \text{in}\ Q_R,
\end{equation}
provided that $M$ and $c$ are chosen appropriately.
This claim will be established in \eqref{claim2} below. 
We  first  fix  a point  $(y, s)$ in $\Rn$. Now for all  $R$ sufficiently  large enough, we  have that  $(y, s) \in Q_{R}$. We would like to emphasize over here that  finally we  let $R \to \infty$.  Therefore, once \eqref{claim} is established, we obtain from it  and the definition of $w$ that
\begin{equation}\label{e:68}
P(u,y, s) \leq  \frac{K'}{R^{1/2}},
\end{equation}
where $K'$ depends on $\ve, (y,s)$ and the bounds of the derivatives of $u$  of order three. By letting  $ R \to \infty$ in \eqref{e:68}, we find  that  
\begin{equation}\label{est}
P(u,y, s) \leq 0. 
\end{equation}
The sought for conclusion thus follows  from the arbitrariness of the point $(y, s)$.

In order to prove the claim \eqref{claim} we argue by contradiction and suppose that there exist $(x_0,t_0)\in \overline Q_{R}$ at which $w$ attains it maximum and for which 
\begin{equation*}
w(x_0, t_0) > K.
\end{equation*}

It follows that at $(x_0,t_0)$ we must have
\begin{equation}\label{expl}
 (\ve^2 + |Du(x_0,t_0)|^2 )^{-1/2} |Du(x_0, t_0)|^2\geq  \frac{1}{2} P (x_0, t_0) \geq \frac{1}{2} w (x_0, t_0)>  \frac{1}{2} K,
\end{equation}
which implies, in particular, that $Du(x_0, t_0) \neq 0$. Therefore, we obtain from \eqref{expl}
\begin{equation}\label{est5}
|Du(x_0,t_0)| \geq  ( 1 + |Du(x_0,t_0)|^2 )^{ - 1/2} |Du(x_0,t_0)|^2 \geq  \frac{1}{2} P (x_0, t_0) >  \frac{1}{2} K.
\end{equation}
On the other hand, since $(x_0,t_0)$ does not belong to the parabolic boundary, from the hypothesis that $w$ has its maximum at such point, we conclude that $w_t(x_0,t_0) \ge 0$ and $Dw(x_0,t_0) = 0$. These conditions translate into
\begin{equation}\label{M1}
P_{t} \geq  \frac{c}{R^{1/2}},
\end{equation} 
and
\begin{equation}\label{M2}
P_{i} =  \frac{M} {R} \frac{x_{0,i}}{ (|x_{0}|^2 + 1 )^{1/2}}.
\end{equation}
Now
\begin{equation*}
(d_{ij}w_{i})_{j} =   (d_{ij} P_{i})_{j} - \frac{M}{R} (d_{ij}  \frac{x_{i}}{ (|x|^2 + 1 )^{1/2}})_{j},
\end{equation*}
where
\begin{equation}\label{calc1}
(d_{ij}P_{i})_{j} = 2 ( \frac{ a_{ij}}{\Lambda} ( \Lambda u_{ki}\ u_{k} - f\ u_{i}))_{j} = 2( a_{ij}\ (u_{k})_{i}\  u_{k})_{j} - 2 (f\ d_{ij}\ u_{i})_{j}.
\end{equation}
After a simplification, \eqref{calc1} equals
\begin{equation*}
2 a_{ij}\  (u_{ki})_{j}\ u_{k} + 2 a_{ij}\ u_{ki}\ u_{kj} - 2 f'\ d_{ij}\  u_{i}\ u_{j} - 2 f\  d_{ij}\ u_{ij} - 2f\ (d_{ij})_{j}\ u_{i}.
\end{equation*}
We notice that
\begin{equation*}
d_{ij}u_{i}u_{j} = \frac{ 2 \phi^{''} u_{i}\ u_{j}\ u_{i}\ u_{j}  + \phi'\ \delta_{ij}\ u_{i}\ u_{j}}{\Lambda} = |Du|^2.
\end{equation*}
Now by using \eqref{f:135} and by  cancelling  the  term  $2f'|Du|^2$, we get that the right-hand side in \eqref{calc1} equals
\begin{equation*}
{\ds 2 \phi' u_{tk}\ u_{k}  + 4 \phi^{''}\ u_{hk}\ u_{h}\ u_{k}u_{t} + 2 a_{ij}\ u_{ki}\ u_{kj} - 2 f d_{ij}\ u_{ij} - 2 f d_{ij, j}\  u_{i}}.
\end{equation*}
 Therefore by using the equation \eqref{e:650},  we obtain 
\begin{align}\label{e:100}
(d_{ij}P_{i})_{j} & = 2 a_{ij}\ u_{ki}\ u_{kj}  +  2\phi^{'}\ u_{tk}\ u_{k} + 4 \phi^{''}\ u_{hk}\ u_{h}\ u_{k}\ u_{t}
\\   
& -  2 \frac{f^2}{\Lambda}  - 2 \frac{ f\ \phi^{'}\ u_t}{\Lambda} - 2fd_{ij,j}\ u_{i}.
\notag
\end{align}
By using the extrema conditions \eqref{M1}, \eqref{M2}, we have the following two conditions at $(x_0, t_0)$ 
\begin{equation}\label{M3}
u_{kh}\ u_{k}\ u_{h} = \frac{f}{\Lambda} |Du|^{2} +  \frac{M}{2 R \Lambda} \frac{x_{h}\  u_{h}}{(|x|^2 + 1)^{1/2}},
\end{equation}
\begin{equation}\label{M4}
2 \Lambda\  u_{kt}\ u_{k}  \geq   2 f u_{t}  + \frac{c}{R^{1/2}}.
\end{equation}
Using the extrema conditions   and by canceling $ 2 \phi^{'} u_{tk}u_{k} $ we  obtain,
\begin{align}\label{11}
(d_{ij} w_{i})_{j}   \geq   & 2a_{ij}\ u_{ki}\ u_{kj}  +   \frac{4 \phi^{''}\ f}{\Lambda} |Du|^{2} u_{t}  -  \frac{2 f^{2}}{\Lambda} - 2 f d_{ij, j}\ u_{i}
\\
 + &  \frac{2 \phi^{''}\ M\ x_h\ u_h\ u_t}{R\ \Lambda\ (|x|^2 + 1)^{1/2}} + \frac{c\ \phi'}{R^{1/2} \Lambda}  -  \frac{M}{R} ( d_{ij} \frac{x_{i}}{ (|x|^2 + 1)^{1/2}})_{j}.
 \notag
 \end{align}
Now  we  have the following  structure equation, whose proof is lengthy but straightforward,
\begin{equation}\label{e:100}
d_{ij, j } u_i = \frac{2 \phi^{''}} {\Lambda} ( |Du|^{2} \Delta u - u_{hk}\ u_h\  u_k).
\end{equation}
Using \eqref{M4} in \eqref{e:100}, we  find
\begin{equation*}
d_{ij, i}\ u_i= \frac{2 \phi^{''} |Du|^2}{\Lambda} (  \Delta u -   \frac{f}{\Lambda} -  \frac{M\ x_h\  u_h}{ 2 R\ |Du|^2\   \Lambda (|x|^2 + 1 )^{1/2}}).
\end{equation*}
Using the equation \eqref{approx},  we have 
\begin{equation*}
2 \phi^{''}\ u_{hk}\ u_h\  u_k + \phi'\  \Delta u  = f + \phi'\  u_t.
\end{equation*}
 Therefore,
\begin{equation}\label{M100}
\Delta\  u = \frac{ f + \phi'\  u_t - 2 \phi^{''}\  u_{hk}\ u_h\  u_k}{\phi'}.
\end{equation}
Substituting the value for $\Delta u$ in \eqref{M100}  and by using the extrema condition \eqref{M4}, we have the following equality  at $(x_0, t_0)$,
\begin{align}\label{e:60}
d_{ij, j}\ u_i  &  = \frac{2 \phi^{''}\ |Du|^2} {\Lambda\ \phi' }\bigg[ f + u_t\ \phi'  - 2 \phi^{''}\ \frac{|Du|^2}{\Lambda} f  - f  \frac{\phi'}{\Lambda}
\\ 
 &  - \frac{  \phi^{''}\ M\ x_h\  u_h } { R \Lambda (|x|^2 + 1 )^{1/2}}-  \frac{M\ x_h\  u_h\ \phi'}{  2 R\ |Du|^2\  \Lambda\  (|x|^2 + 1 )^{1/2}}\bigg].
 \notag
 \end{align}
Using the definition of $\Lambda$  and cancelling terms in \eqref{e:60}, we have that  the right-hand side  in \eqref{e:60}  equals
\begin{equation}\label{e:61}
2 \phi^{''}\frac{ |Du|^2 u_t }{\Lambda} -   \frac{ \phi^{''} M\ x_h\ u_h }{ \Lambda^2\  R\ (|x|^ 2 + 1 )^{1/2}}  -    \frac{ 2 (\phi^{''})^2\ |Du|^2\ M\ x_h\  u_h } {  R\ \Lambda^2\ \phi^{'}\  (|x|^2 + 1 )^{1/2}}. 
\end{equation}
Therefore, by  canceling the terms $4 \phi^{''} f \frac{|Du|^ 2 u_t }{ \Lambda}$ in \eqref{11}, we  obtain the following differential inequality at $(x_0, t_0)$,
\begin{align}\label{e:62}
(d_{ij}w_{i})_{j} \ge &  \frac{c\ \phi'} {  R^{1/2}\ \Lambda} -   \frac{2\ f^2}{\Lambda}   - \frac{M}{R}\  ( d_{ij}\ \frac{x_{i}}{ (|x|^2 + 1)^{1/2}})_{j}  +    \frac{2 \phi^{''}\ M\ x_h\ u_h\ u_t}{R\ \Lambda (|x|^2 + 1)^{1/2}} 
\\  
& + \frac{2 f\ \phi^{''} M\ x_h\ u_h }{ \Lambda^2\ R\ (|x|^ 2 + 1 )^{1/2}} +   \frac{4 f\ (\phi^{''})^2\ |Du|^2\ M\ x_h\  u_h } {R\ \Lambda^2\ \phi^{'}\ (|x|^2 + 1 )^{1/2}} + 2 a_{ij}\ u_{ki}\ u_{kj}.
\notag
\end{align}

Now by using the identity  for $ DP$ in \eqref{P1} above, we have
\begin{equation}\label{e:63}
{\ds u_{ki}\ u_{kj}\ u_{i}\ u_{j}= \frac{( P_k + 2 fu_k ) ^{2}} { 4  \Lambda ^ {2} }}.
\end{equation}
Also,
\begin{equation*}
a_{ij}\  u_{kj}\  u_{ki}  = \phi'\  u_{ik}\ u_{ik}  + 2 \phi^{''}\ u_{ik}\ u_{i}\ u_{jk}\ u_{j}.
\end{equation*}
Therefore, by Schwarz  inequality, we have
\begin{equation*}
a_{ij}\ u_{kj}\ u_{ki} \geq  \phi'  \frac{ u_{ik}\ u_{jk}\  u_i\  u_ j}{  |Du|^2 } + 2 \phi^{''}\  u_{ik}\ u_i\  u_{jk}\ u_{j}   = \frac{\Lambda u_{ik}\ u_{i}\ u_{jk}\ u_{j}}{|Du|^2}.
\end{equation*}
Then, by using \eqref{e:63}  we find
\begin{equation}\label{e:64}
a_{ij}\ u_{kj}\ u_{ki} \geq \frac{(P_{k} + 2fu_k)^{2}}{4 \Lambda |Du|^{2}} =  \frac{  |DP|^2 + 4 f^2 |Du|^2 + 2 f < Du,DP> }{  4 |Du|^ 2 \Lambda }.
\end{equation}
At this point, using \eqref{e:64} in \eqref{e:62}, we can cancel off $\frac{2f^2}{\Lambda}$ and  consequently obtain the following inequality at  $(x_0, t_0)$,
\begin{align}\label{e:65}
(d_{ij}w_{i})_{j} \geq & \frac{c \phi'} {  R^{1/2} \Lambda}   +   \frac{ f < Du, DP> }{  |Du|^2 \Lambda }   - \frac{M}{R}\ ( d_{ij}\ \frac{x_{i}}{ (|x|^2 + 1)^{1/2}})_{j} + \frac{2\ \phi^{''}\ M\ x_h\ u_h\  u_t}{R\ \Lambda (|x|^2 + 1)^{1/2}}  
\\
& + \frac{4 f\ (\phi^{''})^{2}\ |Du|^{2} M\ x_h\  u_h } {   R\ \Lambda^{2}\ \phi' (|x|^2 + 1 )^{1/2}}   +  \frac{2 f\ \phi^{''} M\ x_h u_h }{ \Lambda^2\ R\ (|x|^ 2 + 1 )^{1/2}}.
\notag
\end{align}
By assumption, since  $ w(x_0, t_0) \geq K$, we have that 
\begin{equation*}
 |D u |  \geq \frac{1}{2 R^{1/2}}.
\end{equation*}
Moreover, since  $u$ has  bounded derivatives of  upto order 3, for a fixed $\ve > 0$, we  have that    $\phi' $ and  $\Lambda$ are  bounded from below by a positive constant. Therefore by \eqref{M2},  the term $ \frac{ f < Du, DP> }{  |Du|^2 \Lambda }$ can be controlled from below by   $ - \frac{ M^{''}}{ R^{1/2}}$ where $M^{''}$ depends on $\ve$ and  the bounds of the derivatives of $u$. Consequently, from \eqref{e:65}, we have at $(x_0, t_0)$, 
\begin{equation}
(d_{ij}w_{i})_{j} \geq     \frac{C(c)} { R^{1/2}} -   \frac{L(M)} {R} -  \frac{ M''}{ R^{1/2}}.
\end{equation}
Now in the very first place, if  $c$ is chosen large enough  depending only on $\ve$ and  the bounds of the derivatives of $u$ up to order three,  we would have the following inequality $ \text{at}\ (x_0, t_0)$, 
\begin{equation*}
(d_{ij}w_{i})_{j}   > 0.
\end{equation*}
This contradicts the fact that $w$ has a maximum at  $(x_{0}, t_{0})$.
Therefore, either  $ w (x_{0}, t_{0}) < K$, or    the maximum of $w$  is achieved on the parabolic boundary where $w < 0$.  In either case, for an arbitrary point $(y,s)$ such that $|y| \leq R$,  we  have that 
\begin{equation}\label{claim2}
w(y, s) \leq \frac{1} {R^{1/2}}.
\end{equation}

\end{proof}

\section{Forward gradient  bounds for the reaction-diffusion equation \eqref{e0} in epigraphs}\label{S:fv}

In this section we consider Modica type gradient  bounds for  solutions to the parabolic equation \eqref{e0} in unbounded generalized cylinders of the type $\Om \times [0, T]$. On the ground domain $\Om\subset \Rn$  we assume that it be an epigraph, i.e., that
\begin{equation}\label{l1}
\Om= \{(x', x_n)\in \Rn\mid x'\in \R^{n-1},\  x_n > h(x')\}.
\end{equation}
Furthermore, we assume that $h \in C^{2, \alpha}_{loc}(\R^{n-1})$ and  that
\begin{equation}\label{l2}
||Dh||_{C^{1, \alpha}(\R^{n-1}) } < \infty.
\end{equation}

Before proving the main result of the section we establish a lemma which will be used throughout the rest of the paper.

\begin{lemma}\label{L1}
Let $u$ be a solution to \eqref{e0}, and assume that
\begin{equation}
\underset{G}{\inf}\ |Du| > 0,
\end{equation}
for some open set $G \in \Rn \times \R$. Define
\begin{equation}\label{h1}
P(x,t) \overset{def}{=} P(u,x,t)= |Du(x,t)|^{2}- 2F(u(x,t)).
\end{equation}
Then, we have in $G$ that 
\begin{equation}
(\Delta - \partial_t) P + <B,DP>\ \geq\ \frac{|DP|^2}{2 |Du|^2},
\end{equation}
where $B= \frac{2 F'(u) Du} {|Du|^2}$.
\end{lemma}

\begin{proof}
The  proof of the lemma   follows from computations similar  to that in the proof of Theorem \ref{main1}, but we nevertheless provide details since this lemma will be crucially used in the rest of the paper. We first note that, since $F \in C^{2, \beta}_{loc}$, we have $u \in H_{3+\alpha,loc}$  for some $ \alpha$ which also depends on $\beta$. By using \eqref{e0}, it follows from a  simple computation that
\begin{equation}\label{h10}
(\Delta - \partial_t) P= 2 ||D^{2} u||^2 - 2 F'(u)^2.
\end{equation}
From the definition of $P$, it follows that
\[
DP=  2 D^{2} u Du - 2F'(u) Du.
\]
This gives
\[
4 |D^2 u Du|^2 = |DP + 2 F'(u) Du|^2 = |DP|^2 + 4 F'(u)^2 |Du|^2 + 4 F'(u)<DP,Du>.
\]
Therefore, from Cauchy-Schwartz inequality  we obtain
\[
4 ||D^2 u||^{2} |Du|^2\geq |DP|^2 + 4 F'(u)^2 |Du|^2 + 4 F'(u) <DP,Du>.
\]
By dividing both sides of this inequality by $2 |Du|^2$, and replacing in \eqref{h10}, the desired  conclusion follows.

\end{proof}

We now  state the  relevant  result which is the parabolic analogue of Theorem 1 in \cite{FV}.
\begin{thrm}\label{main2}
Let $\Om\subset \Rn$  be as in \eqref{l1}, with $h$ satisfying \eqref{l2}, and assume furthermore that the mean curvature of $\partial \Om$ be nonnegative. Let  $u$ be a nonnegative  bounded  solution to the following problem
\begin{equation}\label{l3}
\begin{cases}
\Delta u = u_t + F'(u),
\\
u=0\ \text{on}\ \partial \Om \times [0,T],
\end{cases}
\end{equation}
such that
\begin{equation}\label{est2}
|Du|^{2} (x,0) \leq 2 F(u) (x,0).
\end{equation}
Furthermore, assume that $||u(\cdot,0)||_{C^{1, \alpha}(\overline{\Om})} < \infty$.  Then,  the following gradient estimate holds for all $t>0$ and all $x\in \Om$,
\begin{equation}\label{est22}
|Du|^{2} (x,t) \leq 2 F(u) (x,t).
\end{equation}
\end{thrm}

\begin{proof}
Henceforth in this paper for a given function $g :\R^{n-1}\to \R$ we denote by 
\[
\Om_{g} = \{(x',x_n)\in \Rn\mid x'\in \R^{n-1},\  x_n > g(x')\}
\]
its  epigraph. With $\alpha$ as in the  hypothesis \eqref{l2} above, we denote  
\begin{align*}
\mathcal F  = & \bigg\{g \in C^{2, \alpha}(\R^{n-1}) \mid \partial \Om_g\  \text{has nonnegative mean} 
\\
& \text{curvature and}\  ||D g||_{C^{1, \alpha}(\R^{n-1})}\leq ||Dh||_{C^{1, \alpha}(\R^{n-1})}\bigg\}.
\end{align*}

We now note that, given a bounded solution $u$ to \eqref{l3} above, then by Schauder regularity theory (see Chapters 4, 5 and 12 in \cite{Li}) one has 
\begin{equation}\label{d2}
||u||_{H_{1+\alpha}(\overline{\Om} \times [0,T])} \leq C,
\end{equation}
for some universal $C>0$ which also depends on $\Om$ and $||u(\cdot, 0)||_{C^{1, \alpha}(\overline{\Om})}$, and for every $\ve > 0$ there exists $C(\ve)>0$ such that
\begin{equation}\label{d5}
||u||_{H_{2+\alpha}(\overline{\Om} \times [\ve,T])} \leq C(\ve).
\end{equation}
Note that  in \eqref{d5}, we cannot take $\ve=0$, since the compatibility  conditions  at the corner points  need not hold. With $C$ as in \eqref{d2}, we now define
\begin{align*}
\Sigma  = & \bigg\{v \in C^{2, 1}(\overline{\Om}_{g} \times [0,T])\mid \text{there exists}\ g \in \mathcal F\ \text{for which}\ v\ \text{solves \eqref{l3} in}\  \Om_{g} \times [0,T],   
\\
& \text{with}\ 0 \leq v \leq ||u||_{L^{\infty}}, \  ||v||_{H_{1+\alpha}(\overline{\Om}_{g} \times [0, T])} \leq C,\ P(v,x,0)\leq 0 \bigg\}.
\end{align*}
Note that in the definition of $\Sigma$ we have that given any $v \in \Sigma$, there exists  a corresponding  $g^{(v)}$ in $\mathcal F$ such that the assertions in the definition of the class $\Sigma$ hold. Moreover $\Sigma$ is non-empty since $u \in \Sigma$. From now on, with slight  abuse of notation,  we  will  denote  the corresponding $\Om_{g^{(v)}}$  by $\Om_{v}$.

We now set
\[
P_{0}= \underset{v \in \Sigma, (x, t) \in \Om_{v} \times [0,T]}{\sup}\ P(v;x,t).
\]
We note that $P_0$ is finite  because  by the definition of $\Sigma$, every $v \in \Sigma$  has  $H_{1+ \alpha}$  norm bounded  from above by a constant $C$ which  is independent of $v$.  Furthermore, by  Schauder regularity theory we have that  \eqref{d5} holds uniformly for $ v \in \Sigma$ in $\overline{\Om}_{v} \times [0,T]$. Our objective  is to establish that
\begin{equation}\label{obj}
P_{0} \leq 0.
\end{equation}
Assume on the contrary that $P_0>0$. For every $k\in \mathbb N$ there exist $v_{k}\in \Sigma$ and $(x_k, t_{k})\in \Om_{v_k} \times [0,T]$ such that  
\[
P_0 - \frac 1k < P(v_k,x_k,t_k) \le P_0.
\]
By compactness, possibly passing to a subsequence, we know that there exists $t_0\in [0,T]$ such that $t_k\to t_0$. We define
\[
u_{k}(x, t)= v_{k}(x + x_k,t_k).
\]
We then have that $u_{k} \in \Sigma$ and $0 \in \Om_{u_k}$. Moreover, 
\[
P(u_k, 0, t_k) = P(v_k,x_k,t_k) \to P_0.
\]
Now, if we  denote by $g_{k}$ the function corresponding to the  graph of $\Om_{u_k}$, from the fact that $0 \in \Om_{u_k}$ we infer that
\[
g_{k}(0) \leq 0.
\]
We now claim that  $g_{k}(0)$ is bounded. If not, then  there exists a  subsequence such that 
\[
g_{k}(0) \to -\infty.
\]
Moreover since   $||Dg_{k}||_{C^{1,\alpha}}$ is bounded uniformly in $k$, we conclude that  for every $x' \in \R^{n-1}$  
\begin{equation}
g_{k}(x') \to -\infty,
\end{equation}
and the same conclusion holds locally uniformly in $x'$. Since the $u_{k}$'s are  uniformly bounded in $H_{1+\alpha}(\Om_{u_{k}} \times [0,T])$, we have that  $u_{k} \to w_0$ locally  uniformly in $\Rn \times [0,T]$. Note that  this can be justified by taking an extension of $u_k$ to $\Rn \times [0,T]$ such that \eqref{d2} hold in $\Rn \times [0,T]$, uniformly in $k$. Applying \eqref{d5} to the $u_{k}$'s we see that the limit function $w_0$ solves \eqref{l3} in $\Rn \times [0,T]$.  Since by the definition of $\Sigma$ we have $P(u_k,0,0) = P(v_k,x_k,0) \leq 0$, we have that  $t_0>0$, and therefore by \eqref{d2}
 we conclude that $P(w_0,0,t_0) = P_0 > 0$. Moreover, again by \eqref{d2}, we have $P(w_0,0,0) \leq 0$. This leads to a contradiction with the case $p=2$ of Theorem 1.3 established in \cite{BG}. Therefore, the sequence $\{ g_{k}(0)\}$ must be bounded. 

Now since $g_{k}$'s are such that $Dg_{k}$'s have  uniformly bounded $C^{1, \alpha}$ norms,  we conclude by Ascoli-Arzel\`a that there exists $g_{0} \in \mathcal F$ such  that $g_{k} \to g_{0}$ locally uniformly in $\R^{n-1}$. We denote  
\[
\Om_{0} = \{(x', x_n)\in \Rn \mid x'\in \R^{n-1},\  x_n > g_0(x')\}. 
\]
For each $k$,  by taking  an extension $\tilde u_{k}$ of $u_{k}$ to $\Rn \times [0,T]$ such that $\tilde u_{k}$  has  bounded $H_{1+\alpha}$ norm, we have that (possibly on a  subsequence) $\tilde u_k \to u_0$ locally uniformly in $\R^{n}$. Moreover, because of \eqref{d5} applied to $u_k$'s, the function $u_0$ solves the equation \eqref{e0} in $\Om_0 \times (0,T]$. We also note that  since $Dg_k$'s have uniformly bounded $C^{1,\alpha}$ norms, $\partial \Om_{0}$ has nonnegative mean curvature.  Moreover, by arguing as in (33) and (34) in \cite{FV}, we have that $u_0$ vanishes on $\partial \Om_0 \times [0,T]$. Also, it follows  that $P(u_0,x,0) \leq 0$ for $x\in \Om_0$, and therefore $u_0 \in \Sigma$. Arguing by compactness as previously in this proof, we infer that must be $t_0 > 0$, and since $u_0 \in \Sigma$, that 
\[
P_0 = P(u_0,0,t_0)=  \underset{(x,t)\in \Om_0\times [0,T]}{\sup}\ P(u_0,x,t) > 0.
\]
Since $u_0 \geq 0$  and $u_0$ vanishes on $\partial \Om_0 \times [0,T]$, indicating by $\nu$ the inward unit normal to $\partial \Om_0$ at $x$, we have for each $(x,t) \in \partial \Om_0\times [0,T]$
\begin{equation}\label{d6}
\partial_{\nu} u_0 (x, t) \geq 0.
\end{equation}
Given \eqref{d6} and from the  fact that $u_{0}$ is bounded, by arguing as in (36)-(38) in \cite{FV}, it follows that for all $t \in [0,T]$ 
\begin{equation}\label{d7}
\underset{x\in \Om_0}{\inf}\ |Du_0(x,t)|=0.
\end{equation}

Next, we  claim that if for a time level $t > 0$ we have $P(u_0, \overline y, t) =P_0$, then it must be $\overline y \in \partial \Om_0$. To see this, suppose  on the contrary that $\overline y \in \Om_0$. Since  $P_0>0$, this implies that  $|Du_0(\overline y, t)| > 0$. Consider now the set  
\[
U= \{x \in \Om_0\mid P(u_0, x,t)=P_0\}.
\]
Clearly, $U$ is closed, and since $\overline y \in U$  by assumption, we also know that $U \not=\varnothing$. We now  prove that $U$ is open. Since $|Du_0(x, t)|>0$ for every $x \in U$, by Lemma \ref{L1} and the strong maximum principle (we note that since $F \in C^{2, \beta}_{loc}$, we have that $u_0 \in H_{3+\alpha'}$ in the interior for some $\alpha'$ which also depends on $\beta$. Hence, $P(u_0,\cdot,\cdot)$ is a classical subsolution), we conclude that for every $x\in U$ there exists $\delta_x>0$ such that $P(u_0,z,t)= P_0$ for $z \in B(x,\delta_x)$. This implies that $U$ is open.

Since $\Om_0$, being an epigraph, is connected, we conclude that $U = \Om_0$. Now from \eqref{d7} we have that for every fixed $t\in [0,T]$ there exists a sequence $x_j \in \Om_0$ such that $Du_0(x_j,t) \to 0$  as $ j \to \infty$. As a consequence, $\underset{j\to \infty}{\liminf}\ P(u_0,x_j,t)\le 0$. This implies that for large enough $j$ we must have $P( u_0, x_j, t) < P_0$, which contradicts the above conclusion that $U = \Om_0$. Therefore  this establishes the claim that if $P(u_0,\overline y,t) =P_0$, then $\overline y \in \partial \Om_0$. Since $P(u_0,0, t_0)= P_0$, and $P_0$ is assumed to be positive, this implies in particular that $(0, t_0) \in \partial \Om_0 \times (0,T]$. Again, since $P_0>0$ by assumption, we  must have that in \eqref{d6} a strict inequality holds at $(0,t_0)$, i.e.
\begin{equation}\label{d8}
\partial_{\nu} u_0 (0, t_0) > 0.
\end{equation}
This is  because if  the normal derivative  is zero at  $(0, t_0)$, then it must also be $Du_0(0, t_0)=0$ (since  $u_0$ vanishes on the lateral boundary of $\Om_0\times [0,T]$), and this contradicts the fact that $P_0>0$.  

From \eqref{d8} we infer that $Du_0(0, t_0) \neq 0$, and therefore Lemma \ref{L1} implies that, near $(0,t_0)$, the function $P(u_0,\cdot,\cdot)$ is a  subsolution to a  uniformly parabolic equation.  Now,  by an application of the Hopf Lemma (see for instance Theorem 3'  in \cite{LN})  we have that
\begin{equation}\label{d10}
\partial_{\nu} P (u_0,0,t_0) < 0.
\end{equation}
Again  by noting that $u_0$ vanishes   on the lateral  boundary, we have that $\partial_t u_0=0$ at $(0,t_0)$. Therefore, at $(0,t_0)$, the function $u_0$ satisfies the elliptic equation
\begin{equation}\label{d9}
\Delta u_0= F'(u_0).
\end{equation}
At this point, by using the  fact that the mean curvature of $\partial \Om_{0}$ is nonnegative and the  equation \eqref{d9} satisfied by $u_{0}$ at $(0, t_0)$, one can argue as in (50)-(55) in \cite{FV} to reach a  contradiction with \eqref{d10} above. Such contradiction being generated from having assumed that $P_0>0$, we conclude that \eqref{obj} must hold, and this implies the sought for conclusion of the theorem.

\end{proof}

\begin{rmrk}
Note that in the hypothesis of Theorem \ref{main2} instead of $u\ge 0$ we could have  assumed that $\partial_{\nu} u \geq 0$ on $\partial \Om \times [0, T]$.
\end{rmrk}

\begin{rmrk}
We also note that the conclusion in Theorem \ref{main2} holds if  $\Om$ is of the form $\Om= \Om_{0} \times \R^{n-n_0}$  for $1 \leq n_0 \leq n$, where $\Om_{0}$ is a bounded  $C^{2, \alpha}$ domain with nonnegative mean  curvature. The corresponding  modifications in the proof would be as follows.  Let $D$ be the set of  domains which are  all translates of $\Om$. The classes $H$ and $\Sigma$  would be  defined corresponding to $D$  as in \cite{FV}. Then, by arguing as in the proof of Theorem \ref{main2},  we  can assume that  sets $\Om_{u_k} \in D$  are such that  $\Om_{u_k}= p_{k} + \Om$ where  $p_k= (p'_k, 0) \in \R^{n_0} \times  \R^{n-n_0}$. Since $0 \in \Om_{u_k}$ and $\Om_{0}$ is bounded, this implies that $p'_k$ is bounded independent of $k$. Therefore, up to a subsequence, $p_k \to  p_0$ and  $u_k \to u_0$ such that $u_0$ solves \eqref{l3} in $\Om_1= p_0 + \Om$. The rest of the proof remains the same  as that of Theorem \ref{main2}.
\end{rmrk}

\begin{rmrk}\label{i1}
It remains an interesting open  question  whether Theorem \ref{main2} holds for the  inhomogeneous variant of the  normalized $p$-Laplacian evolution studied in \cite{BG}. Note that unlike the case  of $\Rn$, the Hopf lemma  applied to $P$ is a crucial step  in the proof of Theorem \ref{main2}  for which Lemma  \ref{L1} is the key ingredient. As far as we are aware of, an appropriate analogue of Lemma \ref{L1} is not known to be valid for $p \neq 2$,  even in the case when $F=0$.  Therefore, to be able  to generalize Theorem \ref{main2} to the case of  inhomogeneous  normalized $p$-Laplacian evolution as studied in \cite{BG}, lack of  an appropriate subsolution-type  argument (i.e., Lemma \ref{L1}), and a priori $H_{1+\alpha}$ estimates  seem to be the two major obstructions at this point.
\end{rmrk}

\section{Gradient estimates for the reaction-diffusion equation \eqref{e0} in $\Rn \times (-\infty,0]$ and $\Om \times (-\infty,0]$}\label{S:moremods}

In this section we turn our attention to the settings $\Rn \times (-\infty,0]$ and  $\Om \times (-\infty,0]$, where $\Om$ is an epigraph satisfying the geometric assumptions as in the previous section. We investigate the validity of Modica  type gradient estimates in a different situation with respect to that of Section \ref{S:fv}, where such estimates were established under the crucial hypothesis that the initial datum satisfies a similar inequality. We first note that such unconstrained global estimates cannot be expected in $\Rn \times [0,T]$ without any assumption on the initial datum. This depends of the fact that, if at time $t=0$ the initial datum is such that the function defined in \eqref{h1} above satisfies $P(u, x, 0) > 0$ at  some  $x \in \Rn$, then by continuity $P(u, x, t) > 0$ for all $t \in [0,\ve]$, for some $\ve>0$. This justifies our choice of the setting in this section. We now  state our first main result.

\begin{thrm}\label{main3}
Let $u$ be a  bounded solution to \eqref{e0} in $\Rn \times (-\infty, 0]$. Then,  with $P(u,\cdot,\cdot)$ as in \eqref{h1} we have 
\begin{equation}\label{Pu}
P(u, x,t)\leq 0,\ \ \ \ \ \  \text{for all}\ (x, t)\in \Rn \times (-\infty, 0].
\end{equation}
\end{thrm}

\begin{proof}
The proof is inspired to that of Theorem 1.6 in \cite{CGS}. We  define  the   class $\Sigma$ as follows.
\begin{equation}\label{d15}
\Sigma= \{v\mid v\ \text{solves \eqref{e0} in}\ \Rn \times (-\infty, 0],\ \ ||v||_{L^{\infty}} \leq ||u||_{L^{\infty}}\}.
\end{equation}
Note that $u \in \Sigma$. Set
\begin{equation}\label{d16}
P_{0}=\text{sup}_{v \in \Sigma, (x,t) \in \Rn \times (-\infty, 0]}  P(v, x, t).
\end{equation}
Since $F \in C^{2, \beta}_{loc}(\R)$ and the $L^{\infty}$ norm of $v \in \Sigma$ is uniformly bounded by that of $u$, from the Schauder theory we infer all elements $v \in \Sigma$ have uniformly  bounded $H_{3+\alpha}$ norms in $\Rn \times (-\infty,0]$, for some $\alpha$ depending also on $\beta$. Therefore, $P_0$ is bounded.    

We claim that $P_{0} \leq 0$. Suppose, on the contrary, that $P_0>0$. Then, there exists  $v_{k} \in \Sigma$ and corresponding points $(x_k,t_k) \in \Rn \times (-\infty,0]$ such that  $P(v_k,x_k,t_k) \to P_{0}$. Define now
\begin{equation}\label{li}
u_{k}(x,t)= v_{k}(x+x_k, t+ t_k).
\end{equation}
Note that since $t_k \leq 0$, we have that  $u_k \in \Sigma$ and $P(u_k, 0, 0)  = P(v_k,x_k,t_k) \to P_0$. Moreover, since  $u_{k}$'s  have  uniformly  bounded $H_{3+\alpha}$ norms,  for a subsequence, $u_{k} \to u_0$ which belongs to $\Sigma$. Moreover, 
\[
P(u_0, 0, 0) = \underset{(x,t)\in \Rn\times (-\infty,0]}{\sup}\  P(u_0, x, t)= P_0 > 0.
\]
As  before, this implies that $Du_0(0, 0) \neq 0$. Now,  by an application of Lemma \ref{L1}, the strong maximum principle and the connectedness of $\Rn$, we have that $P(u_0,x,0)= P_0$ for all $x \in \Rn$. On the other hand, since $u_0$ is bounded, it follows that
\[
\underset{x\in\Rn}{\inf}\ |Du_0(x,0)|=0.
\]
Then, there exists $x_j \in \Rn$ such that $|Du_0(x_j,0)| \to 0$. However, we have that $P(u_0,x_j,0)=P_0>0$ by assumption which is a contradiction for large enough $j$. Therefore, $P_0 \leq 0$ and the conclusion follows.

\end{proof}

As an application of Theorem \ref{main3} one has the following result on the propagation of zeros whose proof is identical to that of Theorem 1.8 in \cite{CGS} (see also Theorem 1.6 in \cite{BG}).
\begin{cor}\label{main4}
Let $u$ be a bounded solution to \eqref{e0} in $\Rn \times (-\infty, 0]$. If $F(u(x_0,t_0))=0$ for some point $(x_0,t_0)\in \Rn \times (-\infty,0]$, then $u(x,t_0)= u(x_0,t_0)$ for all $x \in \Rn$.
\end{cor}
We also have the following counterpart of Theorem \ref{main3} in an infinite cylinder of the type $\Om \times (-\infty,0]$ where $\Om$ satisfies the hypothesis in Theorem \ref{main2}.

\begin{thrm}\label{main5}
Let $\Om\subset \Rn$  be as in \eqref{l1} above, with $h$ satisfying \eqref{l2}. Furthermore, assume that the mean curvature of $\partial \Om$ be nonnegative. Let $u$ be a nonnegative  bounded  solution to the following problem
\begin{equation}\label{l34}
\begin{cases}
\Delta u = u_t + F'(u)
\\
u=0\ \text{on}\ \partial \Om \times (-\infty,0].
\end{cases}
\end{equation}
Then, we  have that $P(u, x,t)\leq 0$ for all $(x, t)\in \Om \times (-\infty,0]$.
\end{thrm}

\begin{proof}
By Schauder theory we have that
\begin{equation}\label{h13}
||u||_{H_{3+\alpha}(\Om \times (-\infty,0])} \leq C,
\end{equation}
for some $C$ which also depends on $\beta$. We let $\mathcal F$ be  as in the proof of Theorem \ref{main2} and define 
\begin{align*}
\Sigma  = & \big\{v \in C^{2, 1}(\overline{\Om}_{g} \times [0,T])\mid \text{there exists}\ g \in \mathcal F\ \text{for which}\ v\ \text{solves \eqref{e0} in}\  \Om_{g} \times [0,T],   
\\
& \text{with}\ 0 \leq v \leq ||u||_{L^{\infty}}, \  v=0\ \text{on}\ \partial \Om_{g} \times (-\infty,0]\big\}.
\end{align*}
As before, note that in the definition of $\Sigma$ we have that, given any $v \in \Sigma$, there exists  a corresponding  $g^{(v)}$ in $\mathcal F$ such that the assertions in the definition of the class $\Sigma$ hold. With slight  abuse of notation,  we will denote  the corresponding $\Om_{g^{(v)}}$  by $\Om_{v}$. Again by Schauder theory, we  have that any $v \in \Sigma$ satisfies \eqref{h13} in $\Om_{v} \times (-\infty, 0]$,  where the constant $C$ is independent of $v$. We now set,
\[
P_0= \underset{v \in \Sigma, (x,t)\in \Om_v \times (-\infty, 0]}{\sup}\ P(v,x,t).
\]
As before, we claim that  $P_0 \leq 0$. This claim would of course imply the sought for conclusion. From the definition of $\mathcal H$, we note that $P_0$ is bounded.  Suppose, on the contrary, that $P_0 >0$. Then, there exists $v_k$'s and corresponding points $(x_k,t_k)$  such that  $P(v_k,x_k,t_k) \to P_0$. We define,
\begin{equation}\label{d20}
u_k(x, t)= v_k (x+ x_k,t+ t_k).
\end{equation}
Since $t_k \leq 0$, we note that $u_k \in \Sigma$. Now,  by an application of Theorem \ref{main3} and a compactness type argument  as in the proof of  Theorem \ref{main2}, we conclude that if $g_{k}$ is  function corresponding to $\Om_{g_k}= \Om_{u_k}$, then $g_{k}(0)$'s are bounded and since $Dg_{k}$'s have uniformly  bounded  $C^{1, \alpha}$ norms, then $g_{k}$'s are bounded locally uniformly in $\R^{n-1}$. From this point on the proof follows step by step the lines of that of Theorem \ref{main2} and we thus skip pointless repetitions. There exists a $g_0\in \mathcal F$ for which $g_k\to g_0$ locally uniformly in $\R^{n-1}$ as in that proof and we call $\Om_0$ the epigraph of $g_0$. From the uniform Schauder  type  estimates,  possibly passing to a  subsequence, we conclude the existence of a solution $u_0\ge 0$ of   \eqref{e0} in $\Om_{0} \times (-\infty,0]$ such that $\Om_{u_k} \to \Om_{0}$, and $u_k \to u_0$ which solves such that $\partial \Om_{0}$ has nonnegative mean curvature. Moreover, $u_0$  vanishes on the lateral boundary and $P(u_0,0,0)= \sup\ P(u_0,0,0) = P_0$.  The rest of the proof  remains  the same as  that of Theorem \ref{main2}, but with $(0,0)$ in place  of $(0,t_0)$.

\end{proof}
\begin{rmrk}
As indicated in Remark  \ref{i1}, the conclusion of Theorem \ref{main5} remains valid with minor  modifications in the proof when $\Om= \Om_0 \times \R^{n-n_0}$, $1\leq n_0 \leq n,$ where $\Om_0$ is a bounded smooth domain with boundary having nonnegative mean curvature.
\end{rmrk}


\section{Modica  type estimates for reaction-diffusion equations on compact manifolds with nonnegative Ricci tensor}\label{S:ricci}

Let $(M, g)$  be a connected, compact Riemannian  manifold  with Laplace-Beltrami $\Delta_g$, and suppose that the Ricci tensor be nonnegative. In the paper \cite{FV4} the authors established a Modica type estimate for bounded solutions in $M$ of the semilinear Poisson equation 
\begin{equation}\label{ric}
\Delta_{g} u = F'(u),
\end{equation}
under the assumption that $F \in C^{2}(\R)$, and $F\ge 0$. Precisely, they proved that  following inequality holds 
\begin{equation}\label{est130}
|\nabla_{g} u(x)|^2 \leq 2 F(u),
\end{equation}
where $\nabla_{g}$ is the Riemannian  gradient on $M$.  

In this  section, we  prove a  parabolic  analogue of  \eqref{est130}. Our main result  can be  stated  as follows.
\begin{thrm}\label{T:ricci}
Let $M$ be a  connected compact  Riemannian manifold  with $\operatorname{Ric} \ge 0$, and let $u$ be a  bounded  solution to 
\begin{equation}\label{e131}
\Delta_{g} u = u_t +  F'(u)
\end{equation}
on $M \times (-\infty, 0]$ where $F \in C^{2, \beta}(\R)$ and $F\ge 0$. Then,  the following estimate  holds in $M \times (- \infty, 0]$
\begin{equation}\label{est131}
|\nabla_{g} u(x,t)|^2 \leq 2 F(u(x, t)).
\end{equation}
\end{thrm}

\begin{proof}
By  Schauder theory, we have that $u \in H_{3+ \alpha}(M \times (-\infty, 0])$ for some $\alpha$ which additionally depends on  $\beta$. This follows from writing the equation in local coordinates and  by using the compactness of $M$. We next recall the Bochner-Weitzenbock formula, which holds for any $\phi \in C^{3}(M)$
\begin{equation}\label{e132}
\frac{1}{2} \Delta_{g} |\nabla_{g} \phi|^2= |H_{\phi}|^2 + <\nabla_{g} \phi, \nabla_{g} \Delta_{g}\phi> +\operatorname{Ric}_{g}<\nabla_{g} \phi, \nabla_{g} \phi>.
\end{equation}
Here, $H_{\phi}$ is the Hessian of $\phi$ and the square of the Hilbert-Schmidt norm of $H_{\phi}$  is given by
\[
|H_{\phi}|^2= \Sigma_{i} <\bigtriangledown_{X_i} \nabla_{g},  \bigtriangledown_{X_i}, \nabla_{g}>,
\]
where $\{ X_{i} \}$ is a local orthonormal frame.
Moreover, Cauchy-Schwarz inequality gives 
\begin{equation}\label{e133}
|H_{\phi}|^2 \geq |\nabla_{g} |\nabla_{g} \phi||^2.
\end{equation}
See for instance \cite{FSV} for a proof of this fact. Now we define the class 
\[
\mathcal F= \big\{v \mid v\ \text{is a classical solution to}\  \eqref{e131}\ \text{in}\ M \times (-\infty, 0],\ ||v||_{L^{\infty}(M)} \leq ||u||_{L^{\infty}(M)}\big\}.
\]
By the Schauder theory we see as before that for every $v \in \mathcal F$ the norm of $v$ in $H_{3+\alpha}(M \times (-\infty, 0])$ is bounded independent of $v$ for some $\alpha$ which additionally depends on $\beta$. In particular, without  loss of generality, one may assume that the choice  of the exponent $\alpha$ is the same as for $u$. Now, given any $v \in \mathcal F$, we let
\begin{equation}\label{e134}
P(v,x,t)= |\nabla_{g} v(x,t)|^2 - 2 F(v(x,t)).
\end{equation}  
Applying \eqref{e132} we find
\begin{align*}
(\Delta_{g} - \partial_t ) P(v, x,t) & =  2 |H_{v}|^2 + 2 ( <\nabla_{g} v, \nabla_{g} \Delta_{g} v> + \operatorname{Ric}_{g}<\nabla_{g} v, \nabla_{g} v>)   - 
\\
& 2 < \nabla_{g} v, \nabla_{g}v_t>-   2 F'(v)( \Delta_{g} v - v_t)  - 2 <\nabla_{g} v,  \nabla_{g} F'(v)>.
\end{align*}
Using the fact that $v$ solves \eqref{e131}, we obtain
\begin{align*}
(\Delta_{g} - \partial_t ) P(v, x,t) = & 2 |H_{v}|^2 + 2< \nabla_{g} v, \nabla_{g} F'(v)> 
\\
& + 2 \operatorname{Ric}_{g}<\nabla_{g} v, \nabla_{g} v> -  2 F'(v)^2 - 2 <\nabla_{g} v,  \nabla_{g} F'(v)>.
\end{align*}
After cancelling off the term $2< \nabla_{g} v \nabla_{g} F'(v)>$, and by using \eqref{e133} and the fact that the Ricci tensor is nonnegative, we find
\begin{equation}\label{l1bis}
(\Delta_{g} - \partial_t ) P(v, x,t)= 2 |H_{v}|^2 + 2 \operatorname{Ric}_{g}<\nabla_{g} v, \nabla_{g} v> -  2 F'(v)^2 \geq  2  |\nabla_{g} (|\nabla_{g}v|)|^2  - 2 F'(v)^2
\end{equation}
Now from the definition of $P$,
\[
\nabla_{g}P- 2 F'(v) \nabla_{g} v = \nabla_{g}( |\nabla_{g} v|^2 ). 
\]
Therefore,
\[
|\nabla_{g} P|^2 + 4 F'(v)^2 |\nabla_{g} v|^2 - 4 F'(v) < \nabla_{g}v, \nabla_{g}P> = |\nabla_{g} (|\nabla_{g} v|^2)|^2= 4|\nabla_{g} v|^2 |\nabla_{g} (|\nabla_{g}v|)|^2.
\]
By dividing by $2  |\nabla_{g} v|^2$ in the latter equation, we find 
\begin{equation}\label{e136}
\frac{|\nabla_{g} P|^2} {2 |\nabla_{g} v|^2}=  2   |\nabla_{g} (|\nabla_{g}v|)|^2 - 2 F'(v)^2 + 2 \frac{F'(v)}{|\nabla_{g} v|^2 } < \nabla_{g}v, \nabla_{g}P>.
\end{equation}
Combining \eqref{l1bis} and \eqref{e136}, we finally obtain
\begin{equation}\label{l2bis}
(\Delta_{g} - \partial_t ) P  +  2 \frac{F'(v)}{|\nabla_{g} v|^2 } < \nabla_{g}v, \nabla_{g}P> \geq \frac{|\nabla_{g} P|^2} {2 |\nabla_{g} v|^2}.
\end{equation}
The inequality \eqref{l2bis} shows that $P(v, x, t)$ is a  subsolution to a  uniformly  parabolic  equation in any open set where $|\nabla_{g} v| > 0$.
Now  we define
\[
P_0= \underset{v \in \mathcal F, (x, t) \in M \times (-\infty,0])}{\sup}\ P(v, x, t).
\]
Our goal as before is to show that $P_0\leq 0$, from the which  the conclusion of the theorem would follow. Suppose on the contrary that $P_{0} > 0$. Then, there exists $v_{k} \in \mathcal F$ and $(x_k,t_k) \in M \times (-\infty,0]$ such that 
$P(v_k, x_k, t_k) \to P_0$. We define
\[
u_{k}(x, t)= v_{k}(x_k,t+ t_k).
\]
Since $t_k \leq 0$ we have that $u_{k} \in \mathcal F$, and since $M$ is compact, $x_{k} \to x_0$  after possibly passing to a subsequence. Moreover,  $P(u_k,x_0,0) \to P_{0}$. By compactness, we  have that  $u_{k} \to u_0$ in $H_{3+\alpha}$, where $u_0$ is a solution to \eqref{e131}, and $P(u_0, x_0, 0)=P_0 > 0$. Since since $F\ge 0$ this implies that $\nabla_{g} u_0(x_0,0) \neq 0$. By continuity, we see that $\nabla_{g} u_0\not= 0$ in a parabolic  neighborhood of $(x_0,0)$. By \eqref{l2bis} and by the strong maximum principle we infer that $P(u_0, x, 0)=P_0$ in a neighborhood of $x_0$, and since $M$ is connected, we  conclude that for all $x \in M$
\begin{equation}\label{l10}
P(u_0, x, 0)=P_0 > 0.
\end{equation}
Since $u_{0}(\cdot,0)\in C^{1}(M)$ and $M$ is compact, there exists $y_0 \in M$ at which $u_{0}(\cdot,0)$ attains its absolute minimum. At such point one has
\[
\nabla_{g} u_{0}(y_0,0)=0.
\]
Since $F\ge 0$, this implies that
\begin{equation}\label{reference}
P(u_0,y_0,0) \leq 0,
\end{equation}
which is a  contradiction to \eqref{l10}. Therefore $P_0 \leq 0$ and the theorem is proved.
 
\end{proof}

\begin{rmrk}
It remains an interesting question  whether the conclusion of Theorem \ref{T:ricci} ( and for that matter even the corresponding elliptic result in \cite{FV4}) continue  to  hold when $M$ is only  assumed to be complete and not compact. In such a  case, one would need to  bypass the compactness  argument  which  uses  translation  in a  crucial  way ( see for instance \eqref{li}as in the proof of Theorem \ref{main3}). We  intend to come back to this  question in a  future  study.
\end{rmrk}


\section{On a conjecture of De Giorgi and level sets of solutions to \eqref{e0}}\label{S:dg}

In 1978 Ennio De Giorgi formulated the following conjecture, also known as $\ve$-version of the Bernstein theorem: \emph{let $u$ be an entire solution to
\begin{equation}
\Delta u = u^3 - u,
\end{equation}
such that $|u|\le 1$ and $\frac{\p u}{\p x_n} >0$. Then, $u$ must be one-dimensional, i.e., must 
have level sets which are hyperplanes, at least in dimension $n\le 8$}. 

As mentioned in the introduction, the conjecture of De Giorgi has been fully solved for $n=2$ in \cite{GG} and $n=3$ in \cite{AC}, and it is known to fail for $n\ge 9$, see \cite{dPKW}. For $4 \le n \le 8$ it is still an open question. Additional fundamental progress on De Giorgi's conjecture is contained in the papers \cite{GG2}, \cite{Sa}. Besides these  developments, in \cite{CGS} it was established that for entire bounded solutions to 
\begin{equation}\label{plap}
\operatorname{div}(|Du|^{p-2}Du) = F'(u),
\end{equation} 
if the  equality  holds at some point $x_0 \in \Rn$ for the corresponding gradient estimate
\begin{equation}\label{i2}
|Du|^{p}\leq \frac{p}{p-1} F(u),
\end{equation}
then $u$ must be  one dimensional. The result in \cite{CGS} actually regarded a more general class of equations than \eqref{plap}, and in \cite{DG} some further generalizations were presented. We now  establish a  parabolic  analogue of that result in the case $p=2$.

\begin{thrm}\label{main6}
Let $u$ be a bounded solution to \eqref{e0} in $\Rn \times (-\infty,0]$. Furthermore, assume that the zero set of $F$ is discrete. With $P$ as in \eqref{h1} above, if $P(u,x_0,t_0)=0$ for some point $(x_0, t_0) \in \Rn \times (-\infty,0]$, then there exists  $g \in C^{2}(\R)$ such that $u(x,t)= g(<a,x> + \alpha)$ for some $a \in \Rn$ and $\alpha \in \R$. In particular, $u$ is independent of time, and the level sets of $u$ are vertical hyperplanes in $\Rn \times (-\infty,0]$.
\end{thrm}

\begin{proof}
We begin by observing that it suffices to prove the theorem under the hypothesis that $t_0=0$. In fact, once that is done, then if $t_0<0$ we consider the function $v(x,t) = u(x,t+t_0)$. For such function we have $P(v,x,0) = P(u,x,t_0)$ and therefore $v$ satisfies the same hypothesis as $u$, except that $P(v,x_0,0) = 0$. But then we conclude that $v(x,t) = u(x,t+t_0)= g(<a,x> + \alpha)$, which implies the desired conclusion for $u$ as well.

We thus assume without restriction that $P(u,x_0,0) = 0$, and consider the set
\[
A= \{x \in \Rn\mid P(u, x, 0)=0\}.
\]
By the continuity of $P$ we have that $A$ is closed, and since $(x_0,0)\in A$, this set is also non-empty. We distinguish two cases:
\begin{itemize}
\item[\textbf{Case 1:}] There exists $x_1 \in A$ such that $Du(x_1,0) = 0$;
\item[\textbf{Case 2:}] $Du(x,0) \neq 0$ for every $x\in A$.
\end{itemize}

If Case 1 occurs, then from the fact that $P(u,x_1,0)=0$ we  obtain  that $F(u(x_1, 0))=0$. By Corollary \ref{main4} we thus conclude that must be $u(\cdot, 0)\equiv u_0 = u(x_1, 0)$. At this point we observe that, since by assumption $F\ge 0$, and $F(u_0) = 0$, we must also have $F'(u_0) = 0$. Therefore, if we set $v=u-u_0$, then by the continuity of $F''$ and the fact that $u\in L^\infty(\Rn)$,
we have 
\[
|F'(u)| = |F'(v+u_0)| = |F'(v+u_0)- F'(u_0)| \le \int_{u_0}^{v +u_0} |F''(s)| ds \le C |v|.
\]
Since by \eqref{e0} we have $\Delta v - \partial_t v = \Delta u - \partial_t u = F'(u)$,  we see that  $v$ is thus a  solution of the following inequality
\[
|\Delta v - \partial_t v| \leq C|v|.
\]
Since $v(\cdot,0) = 0$, by the backward uniqueness result in Theorem 2.2 in \cite{C}, we  have that $u\equiv u_0$ in $\Rn \times (-\infty,0]$, from which the desired conclusion follows in this case. 

If instead Case 2 occurs, we prove that $A$ is also open. But then, by connectedness, we conclude in such case that $A = \Rn$. To see that $A$ is open fix $x_1\in A$. Since $Du(x_1, 0) \neq 0$, by the continuity of $Du$ we conclude the existence of $r>0$ such that $Du(x,t)\neq 0$ for every $(x,t)\in G = B(x_1,r)\times (-r^2,0]$. By Lemma \ref{L1} above we conclude that $P(u,\cdot,\cdot)$ is a sub-caloric function in $G$. Since by Theorem \ref{main3} we know that $P(u,\cdot,\cdot) \leq 0$, by the strong maximum principle we conclude that $P(u,\cdot,\cdot) \equiv 0$ in $G$. In particular, $P(u,x,0)=0$ for every $x\in B(x_1,r)$, which implies that $A$ is open. 

Since as we have seen the desired conclusion of the theorem does hold in Case 1, we can without loss of generality assume that we are in Case 2, and therefore $Du(x,0)\not = 0$ for every $x\in A = \Rn$. Furthermore, since for $x\in A$ we have $P(u,x,0) = 0$, we also have 
\begin{equation}\label{Puu}
|Du(x,0)|^2=  2 F(u(x,0)),\ \ \ \ \text{for every}\ x\in \Rn.
\end{equation}

Next, we consider the set
\[
K= \big\{(x, t) \in \Rn \times (-\infty, 0]\mid P(u,x,t)=0\big\}.
\]
We note that $K$ is closed and non-empty since by assumption we know that $(x_0,0)\in A$ (in fact, by \eqref{Puu} we now know that $\Rn \times \{0\} \subset K$). Let $(y_1,t_1) \in K$. If $Du(y_1,t_1)=0$, we  can argue as  above (i.e., as if it were $t_1=0$) and conclude  by backward  uniqueness that $u \equiv u(y_1,t_1)$ in $\Rn \times (-\infty,t_1]$. Then, by the forward uniqueness of bounded solutions, see Theorem 2.5 in \cite{LU2}, we  can infer that $u \equiv u(y_1,t_1)$ in $\Rn \times (t_1,0]$. All together, we would have proved that $u \equiv u(y_1,t_1)$ in $\Rn \times (-\infty,0]$ and therefore the conclusion of the theorem would follow.

Therefore from now on,   without loss of  generality, we may assume that $Du$ never vanishes in $K$.  With this assumption in place, if $(y_1,t_1) \in K$, then since $Du(y_1,t_1) \neq 0$,  by continuity there exists $r>0$ such that $Du$ does not vanish in $G = B_{r}(y_1) \times (t_1-r^2,t_1)$. But then, again by Lemma \ref{L1}, the function $P(u,\cdot,\cdot)$ is sub-caloric in $G$. Since $P(u,\cdot,\cdot)\le 0$ in $G$ (Theorem \ref{main5}) and $P(u,y_1,t_1) = 0$ ($(y_1,t_1) \in K$), we can apply the strong maximum principle to conclude that $P\equiv 0$ in $G$. Then, again by connectedness, as in the case when $t_1=0$, we  conclude that $\Rn \times\{t_1\} \subset K$.  In particular, we  have that  $P(u,y_1,t)=0$ when $t \in (t_1 - r^2,t_1]$. Therefore, we  can now  repeat the   arguments  above with $(y_1,t)$ in place  of $(y_1,t_1)$ for each such $t$ and conclude  that $P \equiv 0$ in $\Rn \times ( t_1 - r^2,t_1]$.

We now claim that:
\begin{equation}\label{Puuu}
K = \Rn \times (-\infty,0],\ \ \ \text{or equivalently}\ \ \ P(u,x,t) = 0, \ \text{for every}\ (x,t)\in \Rn \times (-\infty,0].
\end{equation}

Suppose the claim not true, hence  $P \not\equiv 0$ in $\Rn \times (-\infty,0]$. From the above arguments it follows that if for $t_2 <0$ there exists $y_2\in \Rn$ such that $P(u,y_2,t_2) \neq 0$,  then it must be $P(u,x,t_2) \neq 0$ for all $x \in \Rn$. We define
\[
T_0= \text{sup} \{t < 0\mid P(u,\cdot,t) \neq 0\}.
\]
Since we are assuming the claim not true, we must have $\{t< 0\mid P(u,\cdot,t) \neq 0\} \neq \varnothing$, hence $T_0 \le 0$ is well-defined. We first observe that $T_0 < 0$. In fact, since by the hypothesis $(x_0,0)\in K$ and we are assuming that we are in Case 2, we have already proved above the existence of $r>0$ such that $\Rn\times (-r^2,0]\subset K$. This fact shows that $T_0 \le -r^2<0$. Next, we see that it must be $P(u,\cdot,T_0)=0$. In fact, if this were not the case there would exist $y_2\in \Rn$ such that $P(u, y_2, T_0) < 0$. Since $T_0<0$, by continuity we would have that $P(u,y_2,t) < 0$, for all  $t \in [T_0,T_0 + \delta_1)$ for some $\delta_1 > 0$. By the arguments above, this would imply that $P$ never vanishes in $\Rn \times [T_0,T_0 + \delta_1)$, in contradiction with the definition of $T_0$. Since, as we have just seen $P(u,\cdot,T_0)=0$, arguing again as above we conclude that $P \equiv 0$ in $\Rn \times (T_0- r^2,T_0]$ for some $r > 0$. But this contradicts the definition of $T_0$. 

This contradiction shows that $\{t < 0\mid P(u,\cdot,t) \neq 0\} = \varnothing$, hence the claim \eqref{Puuu} must be true. We also recall that we are assuming that $Du$ never vanishes in $K=\Rn \times (-\infty,0]$.

In conclusion, we have that
\begin{equation}\label{e:p1}
|Du|^2= 2 F(u)\ \ \ \ \ \ \text{in}\ \Rn\times (-\infty,0],\ \ \ \text{and}\ \ Du\neq 0.
\end{equation}
At this point we argue as in the proof of Theorem 5.1 in \cite{CGS}, and we let $\nu= H(u)$, where $H$ is a function to be suitably chosen subsequently.
Then, we have that
\[
\Delta \nu - \nu_t= H^{''}(u)|Du|^2 + H'(u) \Delta u - H'(u)u_t.
\]
By using \eqref{e0} and \eqref{e:p1}, we conclude that
\begin{equation}\label{e:p3}
\Delta \nu - \nu_t= 2 H^{''}(u) F(u)  +  H'(u) F'(u).
\end{equation}
Let  $u_0= u(0, 0)$ and define 
\[
H(u)= \int_{u_0}^u (2F(s))^{-1/2} ds.
\]
Since $|Du|(x, t) > 0$ for all $(x, t) \in \Rn \times (-\infty, 0]$, we have from \eqref{e:p1}
that $F(u(x, t)) > 0$.  Therefore, if the zero  set of $F$ is ordered  in the following manner, $a_0 < a_1 < a_2 < a_3< a_4 < ...$, then by connectedness, we have that $F(u(\Rn \times (-\infty, 0])) \subset (a_i, a_{i+1})$ for some $i$. We infer that $H$ is well defined and is $C^{2, \beta}$, and with this $H$ it is easy to check that the right-hand side in \eqref{e:p3} is zero, i.e., $\nu$ is a solution to the heat equation in $\Rn \times (-\infty, 0]$. Moreover, by definition of $H$ and  \eqref{e:p1}, 
\[
|D \nu|^2= H'(u)^2 |Du|^2 =1,
\]
i.e., $D \nu$ is bounded in $\Rn \times (-\infty, 0]$. Since $\nu_i = D_{x_i}\nu$ is a solution to the heat equation for each $i \in {1, ....n}$, by Liouville's theorem in $\Rn \times (-\infty, 0]$ applied to $\nu_i$, we conclude that $D\nu$ is constant, hence $\Delta \nu = 0$. This implies $\nu_t = 0$, hence $\nu$ is time-independent. Hence, there exist $a\in \Rn$ and $\alpha \in \R$ such that $\nu= <a,x> + \alpha$. The desired conclusion now follows by taking $g=H^{-1}$.  This completes the proof of the theorem.

\end{proof}

\section{A parabolic version of the conjecture of De Giorgi}\label{S:dgc}

Motivated by the result in Theorem \ref{main6}, the fact that $\Rn \times (-\infty, 0]$ is the appropriate setting for the parabolic Liouville type theorems, and the crucial role played by them in the proof of the original conjecture of De Giorgi, at least for $n \leq 3$ (see \cite{GG}, \cite{AC}, \cite{GG2}), it is tempting to propose the following parabolic version of De Giorgi's conjecture:

\vskip 0.2in

\noindent CONJECTURE 1: \emph{Let $u$  be a solution in $\Rn \times (-\infty,0]$ to
\[
\Delta u - u_t= u^3-u,
\]
such that $|u|\le 1$, and $\partial_{x_n} u (x, t)  > 0$ for all $(x, t) \in \Rn \times (-\infty,0]$. Then, $u$ must  be one dimensional and  independent of $t$, at least for $n \leq 8$. In other words, for $n \leq 8$ the level sets of $u$ must be vertical hyperplanes, parallel to the $t$ axis}.

However, Matteo Novaga has kindly brought to our attention that, stated this way, the conjecture is not true. There exist in fact eternal traveling wave solutions of the form
\begin{equation}\label{tw}
v(x',x_n,t) = u(x',x_n - ct),\ \ \ \ \ \ c\ge 0,
\end{equation}
for which $\partial_{x_n} u (x)  > 0$. This suggests that one should amend the above in the following way.

\vskip 0.2in

\noindent CONJECTURE 2: \emph{Let $u$  be a solution in $\Rn \times (-\infty,0]$ to
\[
\Delta u - u_t= u^3-u,
\]
such that $|u|\le 1$, and $\partial_{x_n} u (x,t)  > 0$ for all $(x,t)  \in \Rn \times (-\infty, 0]$. Then, $u$ must  be an eternal traveling wave}.

We would still like to regard Conjecture 2 as a parabolic form of De Giorgi's conjecture since, if we also have $u_t \ge 0$, then $u$ must be independent of $t$, and thus we would be back into the framework of the original conjecture of De Giorgi. For interesting accounts of traveling waves solutions we refer the reader to the papers \cite{CGHNR} and \cite{G}. 

In closing, we propose to modify Conjecture 1 by adding to it the assumption that $u_t \ge 0$. With such hypothesis Conjecture 1 would represent a weaker form of Conjecture 2. Nonetheless, it seems to offer some additional challenges with respect to the already remarkable ones presented by the by now classical conjecture of De Giorgi.  We hope that it will stimulate interesting further research.

\end{document}